\documentclass[11pt]{amsart}
\usepackage{graphicx,pdflscape,rotating}
\usepackage[usenames]{color}
\usepackage{amssymb,amscd,amsthm,array}
\usepackage{graphicx,tikz,subfigure}
\usepackage{enumitem}
\usepackage{comment}
\usetikzlibrary{arrows,automata,chains,fit,shapes}
\usepackage{xcolor}
\newcommand{\Z}{\mathbb{Z}}
\newcommand{\N}{\mathbb{N}}

\newcommand{\R}{\mathbb{R}}

\newcommand{\BB}{{\mathcal B}}
\newcommand{\Bvuw}{{\mathcal B}_v^{u,w}}
\newcommand{\Lv}{\mathcal{L}_v}

\newcommand{\LL}{\mathcal{L}}

\newcommand{\DD}{\mathcal D}
\newcommand{\OO}{\mathcal O}

\newcommand{\ti}{{t^{-1}}}
\newcommand{\ai}{{a^{-1}}}

\newcommand{\x}{{\bf x}}
\newcommand{\xx}{{\bf x}}

\newcommand{\wi}{{\bf w}^{(i)}}
\newcommand{\wf}[1]{{\bf w}^{(#1)}}
\newcommand{\wwi}{w^{(i)}}
\newcommand{\z}{{\bf z}}
\newcommand{\y}{{\bf y}}

\newcommand{\sgn}{\ensuremath{\textnormal{sign}}}

\newcommand{\fn}{\left\lfloor\frac{n}{2}\right\rfloor}
\newcommand{\fne}{\frac{n}{2}}
\newcommand{\ord}{\le_{u,w}}

\newcommand{\kx}{{k_{\xx}}}
\newcommand{\ky}{{k_{\y}}}

\newcommand\OOn{\{|\OO_n(N)|\}_{N \in \N}}

\newtheorem{theorem}{Theorem}[section]
\newtheorem{corollary}[theorem]{Corollary}

\newtheorem{lemma}[theorem]{Lemma}

\theoremstyle{definition}

\theoremstyle{remark}

\title{A new proof of the growth rate of the Solvable Baumslag-Solitar Groups}

\author[Jennifer Taback] {Jennifer Taback}
\address{Department of Mathematics, Bowdoin College, Brunswick, ME 04011}
\email{jtaback@bowdoin.edu}

\author[Alden Walker] {Alden Walker}
\address{Center for Communications Research, La Jolla, CA 92121}
\email{akwalke@ccrwest.org}

\thanks{The first author acknowledges support from Simons Foundation
grant 31736 to Bowdoin College.  Both authors thank Moon Duchin, Rob Kropholler and Murray Elder for insightful conversations during the writing of this paper.}

\date{\today}

\begin{document}

\begin{abstract}
We exhibit a regular language of geodesics for a large set of elements of $BS(1,n)$ and show that the growth rate of this language is the growth rate of the group.  This provides a straightforward calculation of the growth rate of $BS(1,n)$, which was initially computed by Collins, Edjvet and Gill in \cite{CEG}.  Our methods are based on those we develop in \cite{TW} to show that $BS(1,n)$ has a positive density of elements of positive, negative and zero conjugation curvature, as introduced by Bar-Natan, Duchin and Kropholler in \cite{BDK}.
\end{abstract}

\maketitle

\section{Introduction}
In this paper we compute the growth rate of the solvable Baumslag-Solitar groups 
\[BS(1,n) = \langle a,t | tat^{-1} = a^n \rangle
\]
for $n \geq 2$.  
The first computation of this growth rate is due to Collins, Edjvet and Gill in \cite{CEG} who additionally exhibit the growth series for the group.  Bucher and Talambutsa in \cite{BT} reprove the results in \cite{CEG} for prime $n$; their methods involve understanding the action of the group on its Bass-Serre tree.  Their goal is to show that the minimal exponential growth rates of the solvable Baumslag-Solitar group $BS(1,n)$ and the lamplighter group $L_n = \Z_n \wr \Z$ coincide for prime $n>2$ but differ for $n=2$.  In both \cite{BT} and \cite{CEG}, as well as in the proofs below, the rate of growth is computed to be the reciprocal of a root of a particular polynomial.  When $n>2$ is even our polynomials are much simpler than those in \cite{CEG}, and they match those in both references in the remaining cases.

Our proofs rely on a series of techniques developed by the authors in \cite{TW} to produce a geodesic representative for an element of $BS(1,n)$ given in a standard normal form.
In \cite{TW} we prove that $BS(1,n)$ has a positive density of elements of positive, negative and zero conjugation curvature, as introduced by Bar-Natan, Duchin and Kropholler in \cite{BDK}.  
A direct consequence of our methods for understanding geodesic words is the computation of the growth rate contained below.
We present a concise version of the arguments in \cite{TW} and refer the reader to that paper for more details.

Briefly, our approach is as follows.  
Given any element of $BS(1,n)$ in a standard normal form, we
parametrize a set of paths representing the element which we show contains a geodesic.  
These paths
come in four basic ``shapes.''
Focusing on geodesic paths of one particular shape, we prove that the set of all paths of this shape form a regular language, and exhibit a finite state automaton which accepts it.
This allows us to analyze the growth rate of this set, which we show to be identical to the growth rate of the group $BS(1,n)$.  
This is analogous to the work of Brazil in \cite{Brazil}, who follows the same outline to show that $BS(1,n)$ has rational growth for all $n>1$.  
He remarks that
one should be able to use his methods to calculate the exact growth rate but does not do so.

Also of interest for $BS(1,n)$ is the conjugacy growth rate.  
In recent work, Ciobanu, Evetts and Ho in \cite{CEH} show that the conjugacy growth rate for BS(1,n) is identical to the standard growth rate, using the presentation given above.  If $G$ is a group with finite generating set $S$, the conjugacy growth function measures the number of conjugacy classes intersecting the ball of radius $m$ in $\Gamma(G,S)$.  They show that the corresponding generating function, called the conjugacy growth series, is transcendental.  These results provide positive evidence towards two conjectures: first, that the conjugacy and standard growth rates are identical in finitely presented groups, and second, that only virtually abelian groups have rational conjugacy growth series.

\section{Representations of integers and geodesic paths}
\label{section:min_rep}

\subsection{Background and approach}
For $n \in \N$ with $n > 1$, the solvable Baumslag-Solitar group $BS(1,n)$ has presentation $$BS(1,n) = \langle a,t | tat^{-1} = a^n \rangle.$$  
We consider elements of $BS(1,n)$ in the standard normal form, namely each $g \in BS(1,n)$ can be written uniquely as $t^{-u}a^vt^w$ where  $u,v,w \in \Z$ and $u,w \geq 0$, with the additional requirement that if $n|v$ then $uw=0$. 
If $n|v$ but $uw \neq 0$ then the group relator can be applied to simplify the normal form expression.  When we write $g=t^{-u}a^vt^w$ we will assume that these conditions are satisfied.

Our approach to finding geodesic words representing elements of $BS(1,n)$ builds on \cite{EH}, where it is shown that any geodesic takes one of a small number of prescribed forms.
We create a vector from the exponents of the generator $a$ in any one of these forms, which is related to the horizontal distance traveled by the path in the Cayley graph $\Gamma(BS(1,n),\{a,t\})$.
We then develop criteria to determine when a vector of exponents corresponds to a geodesic path.  This strategy is described more fully in the sections below, with proofs of the statements provided in \cite{TW}.

\subsection{The digit lattice}
In order to produce a geodesic representative of a given group
element \hbox{$g = t^{-u}a^vt^w$}, 
we begin with an investigation of finite length vectors with integer entries having bounded absolute value.
We describe an algorithm to translate this vector, which we also refer to as a digit sequence, into a path in $\Gamma(BS(1,n),\{a,t\})$.  
Our goal is to impose simple conditions on the sequence of digits so that the resulting path will be geodesic.

A digit sequence has an associated integer number, 
which  becomes the exponent of $a$ in the standard normal form for any path associated with this digit sequence.
We formalize the concept of digit sequences
using the direct sum $\bigoplus_{i\in\N} \Z$, where we take the convention that
$0 \in \N$.  
Given a vector
$\xx = (x_0, x_1, \dots) \in \bigoplus_{i\in\N}\Z$, 
define the function
$\Sigma:\bigoplus_{i\in\N}\Z \rightarrow \R$ by 
\[
\Sigma(\xx)=\sum_{i \in \N} x_in^{i}.
\]

For any $v \in \Z$, let $\Lv = \Sigma^{-1}(v)$ be the
set of vectors $\xx \in \bigoplus_{i\in\N}\Z$ with
$\Sigma(\xx) = v$.  
We establish some notation for the remainder of this paper.  Given a vector $\x$,
\begin{itemize}[itemsep=5pt]
    \item the coordinates of $\xx$ will be written with a matching non-bold letter, for example, $x_i$,
    \item $\kx$ will denote the index of the final nonzero coordinate in $\x$, and 
    \item the length of the vector is $\kx+1$.
\end{itemize}
Vector entries will be called either ``coordinates" or ``digits." 
These vectors, while constructed in $\bigoplus_{i\in\N}\Z$ or $\Lv$, have a finite number of non-zero entries; we list only those non-zero entries and write, for example, 
$\xx = (x_0, \dots, x_{\kx})$, assuming that $x_i = 0$
for $i > \kx$ and $x_{\kx} \neq 0$.

Let $\LL_0 = \Sigma^{-1}(0)$; in \cite{TW} we show that $\LL_0$ is a lattice spanned by the following set of vectors.
Define vectors $\{ \wi \}_{i \in \N}$ whose coordinates $\wwi_j$ are given by
\[
\wwi_j = \left\{ \begin{array}{ll} 1  & \textnormal{if $j=i+1$} \\
                                   -n & \textnormal{if $j=i$} \\
                                   0  & \textnormal{otherwise}
                 \end{array}\right.
\]
That is,
\[
\wi = (0, \ldots\, , 0\, ,\, \wwi_i, \wwi_{i+1}, 0, \ldots) =
(0, \underset{\ldots}{\ldots}, \underset{i-1}{0}, \underset{i}{-n}, \underset{i+1}{1}, \underset{i+2}{0},\, \underset{\ldots}{\ldots})
\]
where we indicate the index of each entry in the second expression.
In \cite{TW} we show that $\mathcal{L}_v$ is an affine lattice, for $v \in \Z$.
We use these vectors $\{ \wi \}_{i \in \N}$ to describe a deterministic algorithm which produces a geodesic representative for $g \in BS(1,n)$.

Given a group element  $g = t^{-u}a^vt^w$, we define a map
$\eta_{u,v,w}:\Lv \to \{a^{\pm},t^{\pm}\}^*$ which takes
a vector $\xx = (x_0, \dots, x_\kx) \in \Lv$ to a word
$\eta_{u,v,w}(\xx)$
representing $g$ in the following way:
\[
\eta_{u,v,w}(\xx) = 
\left\{
\begin{array}{lll}
t^{-u} a^{x_0} t a^{x_{1}} \cdots t a^{x_{\kx}}t^{w-\kx}
& \textnormal{if $\kx \leq w$} & \textnormal{(shape 1)} \\
t^{\kx-u} a^{x_{\kx}} \ti a^{x_{\kx-1}} \cdots \ti a^{x_0} t^w
& \textnormal{if $w < \kx \leq u$} & \textnormal{(shape 2)} \\
t^{-u}a^{x_0}ta^{x_{1}} \cdots ta^{x_{\kx}} t^{w-\kx}
& \textnormal{if $u \leq w<\kx$} & \textnormal{(shape 3)} \\
t^{\kx-u}a^{x_{\kx}}\ti a^{x_{\kx-1}} \cdots  \ti a^{x_{0}} t^{w}
& \textnormal{if $w<u<\kx$} & \textnormal{(shape 4).} 
\end{array}\right .
\]
Paths of shapes 1 and 3 and shapes 2 and 4 have identical expressions up to the signs of certain exponents.
Following \cite{TW} we additionally denote those geodesics of shape $1$ for which $\kx < w$ as geodesics of {\em strict} shape 1.

We now show that the length of each path above is given by one
of two expressions.  Here, $| \cdot |$ denotes the actual length of the
given path, not the word length with respect to the generating set $\{a^{\pm 1},t^{\pm 1} \}$ in the group of the element it represents.  We repeat the proof from \cite{TW} as this lemma is crucial to subsequent results.

\begin{lemma}[\cite{TW}, Lemma 3.7]\label{lemma:length_formula}
For $\xx = (x_0, \dots, x_\kx) \in \Lv$, we have
\[
|\eta_{u,v,w}(\xx)| = 
\left\{
\begin{array}{lll}
\Vert \xx \Vert_1 + u + w & \textnormal{if $\kx \leq \max(u,w)$} & \textnormal{(shapes 1 and 2)} \\ 
\Vert \xx \Vert_1 + 2\kx - |u-w| & \textnormal{otherwise} & \textnormal{(shapes 3 and 4)}
\end{array}\right. .
\]
\end{lemma}
\begin{proof}
To prove the lemma, we sum the absolute values of the exponents in the above expressions for $\eta_{u,v,w}(\xx)$.  Accounting for the signs of the expressions, the first formula follows immediately.

For the second case, we compute the length of a path of shape 3:
\[
|\eta_{u,v,w}(\xx)| = \Vert \xx \Vert_1 + u + \kx + \kx - w
\]
and for shape 4:
\[
|\eta_{u,v,w}(\xx)| = \Vert \xx \Vert_1 + (\kx-u) +  \kx + w.
\]
Considering the relative magnitudes of $u,w$ and $\kx$, we see that the two expressions combine into the second formula of the lemma.
\end{proof}

When $\kx = \max(u,w)$, it follows easily that the two expressions for word length given in Lemma~\ref{lemma:length_formula} agree.
To verify this, when $\kx=\max(u,w)$ we have
\begin{align*}
2\kx - |u-w| & = 2\max(u,w) - |u-w| \\
           & = 2\max(u,w) - \max(u,w) + \min(u,w) \\
           & = u + w
\end{align*}

The following lemma follows directly, and a proof is provided in \cite{TW}.

\begin{lemma}[\cite{TW}, Lemma 3.9]\label{lemma:minimal_is_geodesic}
If $\xx \in \Lv$ is such that $|\eta_{u,v,w}(\xx)|$ is minimal,
then $\eta_{u,v,w}(\xx)$ is a geodesic representing the
group element $g = t^{-u}a^vt^w$.
\end{lemma}

If we are given $g = t^{-u}a^vt^w$ and want to find a geodesic
for $g$, then by Lemma~\ref{lemma:minimal_is_geodesic}, it suffices
to find a vector $\xx \in \Lv$ such that 
$|\eta_{u,v,w}(\xx)|$ is minimal; we will refer to such an
$\xx$ as a minimal vector.  
As $\Lv$ is an affine lattice, 
this is equivalent to minimizing 
$|\eta_{u,v,w}(\xx+\z)|$, where $\xx \in \Lv$ is any
vector and $\z \in \mathcal{L}_0$.  Lemma~\ref{lemma:reduce_to_box} shows that some vectors $\xx$ are easily altered in this way to reduce $|\eta_{u,v,w}(\xx)|$.  We will refer to the change from $\x$ to $\x+\z$ in this way as {\em reducing} $\x$.

For $n \geq 3$, let $\Bvuw \subseteq \Lv$ be defined as the set of 
$\xx = (x_0, \dots, x_\kx) \in \Lv$ satisfying the following conditions.

\begin{enumerate}
\item If $i<\kx$, then $|x_i| \le \fn$.
\item If $i=\kx < \max(u,w)$, then $|x_i| \le \fn$.
\item If $i=\kx \ge \max(u,w)$, then $|x_i| \le \fn + 1$.
\end{enumerate}
When $n=2$, we define $\Bvuw$ as above, replacing the third inequality with $|x_i| \le \fn + 2$.

Note that the entries of vectors in $\Bvuw$ are uniformly bounded in absolute value by $\fn$ with the exception of the final coordinate, which in some cases can be slightly larger in absolute value.
With the goal of finding minimal vectors in $\Bvuw$, this modified bound on the final digit results from certain examples where a larger final digit produces a shorter path.

In Lemma~\ref{lemma:reduce_to_box} below we show that given $\x \in 
\Lv$, we can find a vector $\y \in \Bvuw \subseteq \Lv$ so that
$|\eta_{u,v,w}(\y)| \le |\eta_{u,v,w}(\xx)|$.
Since $\eta_{u,v,w}(\xx)$ and $\eta_{u,v,w}(\y)$
represent the same group element, this implies that finding a geodesic for a
group element is equivalent to searching for a minimal vector within $\Bvuw$.  For such $\x$ and $\y$, we will
write $\xx \ord \y$ to mean
that $|\eta_{u,v,w}(\y)| \le |\eta_{u,v,w}(\xx)|$.  Note that
although $\ord$ is transitive, it is not a partial order because it is
not antisymmetric.  However, it still makes sense to refer to vectors
as being minimal with respect to the relation.

The following lemma follows immediately.
\begin{lemma}[\cite{TW}, Lemma 3.10]
\label{lemma:reduce_to_box}
If $\xx  \notin \BB_v^{u,w}$, then there exists
$\z \in \mathcal{L}_0$ so that $\xx + \z \in \BB_v^{u,w}$
and 
\[
\xx + \z \ord \xx.
\]
Consequently, if $\xx \in \Bvuw$ is minimal in $\Bvuw$, then $\eta_{u,v,w}(\xx)$
is geodesic.
\end{lemma} 

In the sections below,we will give
some simple conditions to certify that $\x$ is minimal.

\subsection{Minimal vectors for $n$ odd}
\label{section:geodesics_n_odd}

Let $g = t^{-u}a^vt^w \in BS(1,n)$ for $n$ odd.
Lemma~\ref{lemma:reduce_to_box} shows that $\Bvuw$ is nonempty
and that if $\xx \in \Lv$ is a
minimal vector in $\Bvuw$, then $\eta_{u,v,w}(\xx)$ is geodesic for $g$.
The next lemma shows that when $n$ is odd, the set $\Bvuw$ contains at most two vectors.

\begin{lemma}[\cite{TW}, Lemma 3.13]
\label{lemma:odd_box}
Let $n \geq 3$ be odd and $\xx \in \Bvuw$.
\begin{enumerate}[itemsep=5pt]
\item If $k_\xx < \max(u,w)$, then $|\Bvuw| = 1$.
\item If $k_\xx \ge \max(u,w)$, then $|\Bvuw| \le 2$.
If $|\Bvuw| = 2$, then $\Bvuw$ has the form
\[
\Bvuw = \{\xx, \xx + \epsilon{\bf w}^{(k_\xx)}\},
\]
where $\epsilon \in \{-1,1\}$.  Moreover, $\y\in\Bvuw$ is not minimal
if and only if $\ky > \max(u,w)$ and the final digits of $\y$ are
$(\delta\fn, -\delta)$, 
where $\delta \in \{\pm 1\}$.
\end{enumerate}
\end{lemma}

Lemma~\ref{lemma:odd_box} presents a direct algorithm for producing a geodesic representative
of $g = t^{-u}a^vt^w$ in $B(1,n)$ when $n$ is odd.  Begin with any vector in $\Lv$; if it does not lie in $\Bvuw$,
reduce its digits as described above so that it does. 
Then inspect the final two digits to assess minimality, adding a basis vector as specified by the theorem if necessary.

\subsection{Minimal vectors for $n$ even}
\label{section:geodesics_n_even}

When $n$ is even, $\Bvuw$, as defined above, may contain more than two minimal vectors.
In order to make a consistent choice among them, we add a constraint on the absolute values of the digits.  
For $n$ even, we will say that $\x$ is minimal
if $|\eta_{u,v,w}(\x)|$ is minimal and the vector of absolute values of $\x$
is lexicographically minimal among all such vectors.  For this lexicographic order,
smaller-index digits are considered more significant.

As there may be many more vectors in $\Bvuw$ to consider when $n$ is even, the question of deciding
whether a vector is minimal is more complicated.  
For us it will suffice to
characterize minimality for a subset of all vectors in $\Bvuw$, namely, those vectors $\x \in \Bvuw$ with $\kx < w$.
These vectors correspond to paths $\eta_{u,v,w}(\x)$ of strict shape $1$, and Lemma~\ref{lemma:adjacent_digits} below
describes when such vectors are minimal.

We subdivide this case by $n=2$ and $n>2$ as the definition of $\Bvuw$ is slightly different.
Our goal is to exhibit simple local conditions to determine whether a vector $\x \in \Bvuw$ with $\kx <w$ is minimal.
Such a minimal vector will correspond to a geodesic 
$\eta_{u,v,w}(\x)$ of strict shape $1$.
We will compute the growth rate of the language of geodesics of strict shape 1, and show that this is the same as the growth rate of $BS(1,n)$.

We condense Lemmas 3.22 and 3.29 of \cite{TW} into the following lemma.
While the result is identical for $n=2$ and $n>2$, the methods of proof are slightly different.  We refer to reader to \cite{TW} for these two proofs.

\begin{lemma}[\cite{TW}, Lemmas 3.22 and 3.29] \label{lemma:adjacent_digits}
Let $n$ be even and $\x \in \Bvuw$ with $\kx < \max(u,w)$.
Then $\x$ is not minimal if and only if one of the following holds, for $\delta \in \{\pm 1\}$.
\begin{itemize}[itemsep=5pt]
\item There are two adjacent digits
of the form $(\delta\fne,\delta\fne)$.
\item There are two adjacent digits 
of the form $(\delta\fne,x_i)$ with $\sgn(x_i) = -\sgn(\delta)$.
\end{itemize}
\end{lemma}

\section{Regular languages}
\label{section:regular_languages}

Given $u,w$, and $\x$ with $\kx <\max(u,w)$, 
Lemmas~\ref{lemma:odd_box} and ~\ref{lemma:adjacent_digits} provide a straightforward
way to determine whether $\x \in \Bvuw$ is minimal,
that is, whether $\eta_{u,v,w}(\x)$ is a geodesic, by examining the
digits of $\x$.  
Recall that if $\kx < \max(u,w)$,
we say that $\eta_{u,v,w}(\x)$ has \emph{strict} shape 1.

In this section, we show that the set of vectors
$\x$ for which there are $u,w$ so that $\eta_{u,v,w}(\x)$
is geodesic and has strict shape 1 forms a regular language.  We then
show that the set of geodesics of strict shape 1 is also
a regular language and construct automata
accepting both of these languages.  
This will allow us,
in Section~\ref{section:bs1n_growth},
to count the geodesics of strict shape 1 with a given length
and determine the growth rate of $BS(1,n)$.

\subsection{The language of strict shape 1 vectors}
\label{section:shape_1_vectors}

Let $\DD_n$ be the language of 
minimal vectors $\x \in \Bvuw$ for some $u,v,w$
where $\eta_{u,v,w}(\x)$ has strict shape $1$.

\begin{lemma}\label{lemma:lex_min_reg}
For all $n \geq 2$, the language $\DD_n$ is regular. These languages are accepted by the finite state automata shown in
Figure~\ref{fig:fsa_Dn_odd} and~\ref{fig:fsa_Dn_even}.
\end{lemma}
\begin{proof}
Note that $\DD_n$ is the subset of words in
$\left\{-\fn, \ldots, \fn\right\}^*$ which satisfy
no condition of Lemmas~\ref{lemma:odd_box},
or~\ref{lemma:adjacent_digits} and do
not end with a $0$.
The conditions of Lemmas~\ref{lemma:odd_box} and
~\ref{lemma:adjacent_digits} are local conditions
which are therefore regular, and the finite automata shown in Figures~\ref{fig:fsa_Dn_odd} and~\ref{fig:fsa_Dn_even}
accept only digit strings which do not satisfy any of them.
\end{proof}
We remark that $\DD_n$ is the language of digit sequences
which are minimal for some $u,v,w$. 
Thus the pattern
$(\fn, -1)$ is permitted at the end of a sequence,
because if  $w$ is sufficiently large, this pattern can exist in a minimal vector.

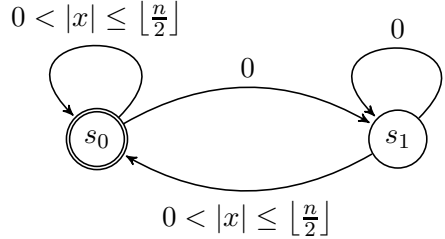
\begin{figure}[ht!]
\tikzset{every state/.style={minimum size=2em}}
\begin{center}
\begin{tikzpicture}[->,>=stealth',shorten >=1pt,auto, node distance=4cm, semithick]

\node[state,accepting] (A) {$s_0$};
\node[state] (B) [right of =A] {$s_1$};

\path(A)
edge [loop] node[above] {$0 < |x| \leq \fn$} (A)
edge [bend left] node[above]{$0$} (B);

\path(B)
edge [loop] node[above] {$0$} (B)
edge [bend left] node[below] {$0 < |x| \leq \fn$} (A);

\end{tikzpicture}
\end{center}
\caption{A finite state automaton accepting the regular language
$\DD_n$ when $n$ is odd. An edge label for a range of $x$
values represents that many single edges with labels in the
appropriate interval.
The state $s_0$ is the start and accept state.}
\label{fig:fsa_Dn_odd}
\end{figure}

\begin{figure}[ht!]
\tikzset{every state/.style={minimum size=2em}} 
\begin{center}
\resizebox {\columnwidth} {!} {
\begin{tikzpicture}[->,>=stealth',shorten >=1pt,auto, node distance=4cm,semithick]

\node[state,accepting] (A) {$s_0$};
\node[state,accepting] (B) [right of =A] {$s_1$};
\node[state, draw=none] (BB) [right of=B] {};
\node[state,accepting] (C) [left of = A] {$s_2$};
\node[state, draw=none] (CC) [left of=C] {};
\node[state] (D) [below of =A] {$s_3$};

\path(A)
edge [loop] node[above] {$0 < |x| < \frac{n}{2}$} (A)
edge [bend right] node[right]{$0$} (D)
edge [bend left] node[above]{$\frac{n}{2}$} (B)
edge [bend right] node[above]{$-\frac{n}{2}$} (C);

\path(B)
edge node[below]{$0 < x < \frac{n}{2}$} (A)
edge node[below]{$0$} (D)
edge [shorten >=75pt, dashed] node[above] {$x < 0$} (BB)
edge [shorten >=75pt, dashed] node[below] {$x = \frac{n}{2}$} (BB);

\path(C)
edge node[below] {$-\frac{n}{2} < x < 0$} (A)
edge node[below]{$0$} (D)
edge [shorten >=75pt, dashed] node[above] {$x > 0$} (CC)
edge [shorten >=75pt, dashed] node[below] {$x = -\frac{n}{2}$} (CC);

\path(D)
edge [bend right] node[right,pos=.77] {$0 < |x| < \frac{n}{2}$} (A)
edge [bend right] node[below, pos=.25]{$\frac{n}{2}$} (B)
edge [bend left] node[below, pos=.25]{$-\frac{n}{2}$} (C)
edge [loop below] node[below] {$0$} (D);

\end{tikzpicture}
}
\end{center}
\caption{A finite state automaton accepting the regular language
$\DD_n$ when $n$ is even. An edge label for a range of $x$
values represents that many single edges with labels in the
appropriate interval.
The start state is $s_0$, and all states except $s_3$ are accept states.
The dashed edges terminate in a fail state.}
\label{fig:fsa_Dn_even}
\end{figure}
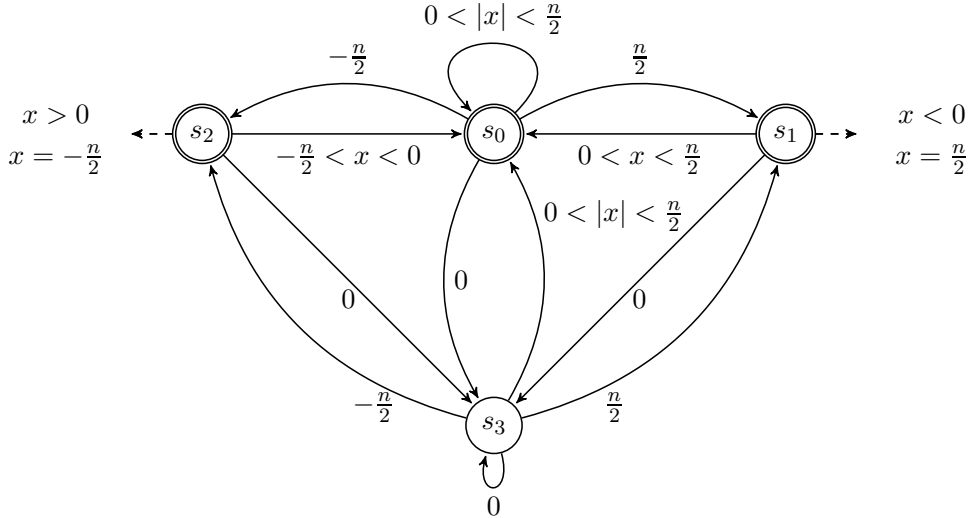

For technical reasons, it will be helpful in Section~\ref{section:shape_1_paths}
to consider a language very closely related to $\DD_n$: the
language of strict shape 1 vectors which are allowed to end with a string of $0$ digits.
That is, the language of vectors which satisfy no condition of Lemma~\ref{lemma:odd_box} or 
\ref{lemma:adjacent_digits}.  This simply relaxes the last
condition from the definition of $\DD_n$.  We denote this new language by $\DD_n'$.
Note that $\DD_n \subseteq \DD_n'$.  Finite state automata which accept $\DD_n'$
are shown in Figure~\ref{fig:fsa_Dnprime}.
These simpler automata are obtained by merging the state
keeping track of the $0$ digit (that is, $s_1$ for $n$ odd and $s_3$ for $n$ even) into the start
state $s_0$.

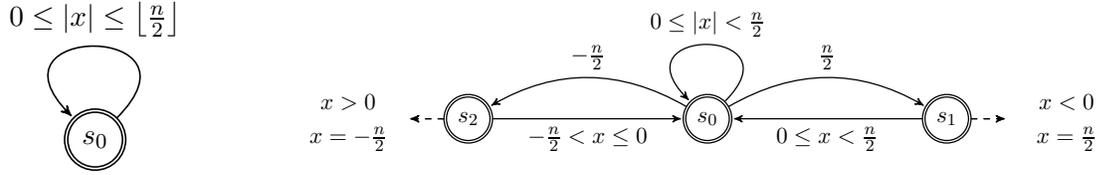
\begin{figure}[ht!]
\tikzset{every state/.style={minimum size=2em}}
\begin{center}
\begin{minipage}{0.17\columnwidth}
\centering
\begin{tikzpicture}[->,>=stealth',shorten >=1pt,auto, node distance=4cm, semithick]

\node[state,accepting] (A) {$s_0$};

\path(A)
edge [loop] node[above] {$0 \le |x| \leq \fn$} (A);

\end{tikzpicture}
\end{minipage}
\begin{minipage}{0.82\columnwidth}
\centering
\resizebox {\columnwidth} {!} {
\begin{tikzpicture}[->,>=stealth',shorten >=1pt,auto, node distance=4cm,semithick]

\node[state,accepting] (A) {$s_0$};
\node[state,accepting] (B) [right of =A] {$s_1$};
\node[state, draw=none] (BB) [right of=B] {};
\node[state,accepting] (C) [left of =A] {$s_2$};
\node[state, draw=none] (CC) [left of=C] {};

\path(A)
edge [loop] node[above] {$0 \le |x| < \frac{n}{2}$} (A)
edge [bend left] node[above]{$\frac{n}{2}$} (B)
edge [bend right] node[above]{$-\frac{n}{2}$} (C);

\path(B)
edge node[below]{$0 \le x < \frac{n}{2}$} (A)
edge [shorten >=75pt, dashed] node[above] {$x < 0$} (BB)
edge [shorten >=75pt, dashed] node[below] {$x = \frac{n}{2}$} (BB);

\path(C)
edge node[below] {$-\frac{n}{2} < x \le 0$} (A)
edge [shorten >=75pt, dashed] node[above] {$x > 0$} (CC)
edge [shorten >=75pt, dashed] node[below] {$x = -\frac{n}{2}$} (CC);
\end{tikzpicture}
}
\end{minipage}

\end{center}
\caption{Finite state automata accepting the regular language
$\DD_n'$ when $n$ is odd (left figure) and even (right figure).  The state
$s_0$ is the start state and all states are accept states.}
\label{fig:fsa_Dnprime}
\end{figure}

\subsection{The language of strict shape 1 geodesics}
\label{section:shape_1_paths}

The language $\DD_n$ is a subset of the full set of minimal digit sequences.  Note that $\DD_n$
is \emph{not} a language of geodesic words in $BS(1,n)$;
if $\x$ is an accepted string in $\DD_n$ then $\eta_{u,v,w}(\x)$ is a word in $BS(1,n)$, 
where for appropriate choices of $u,w$, it follows that $\eta_{u,v,w}(\x)$ is a geodesic.
Let $\OO_n$ be the language of strict shape $1$ geodesics in $BS(1,n)$.
Using $\DD_n'$ and the finite state automaton accepting it, 
we will show that $\OO_n$ is regular and exhibit a finite state automaton accepting it.

For simplicity, we will abuse notation and write $\DD_n'$ for both the language
and the finite state automaton which accepts it.
We now derive a finite state automaton which accepts $\OO_n$  from
$\DD_n'$. We will use $\OO_n$ to refer to both the language of strict shape
$1$ geodesics and the finite state automaton which accepts this language.

One nice feature of
$\DD_n$ and $\DD_n'$ is that the number of
states is independent of $n$.  This is not the case for
$\OO_n$, so we cannot exhibit a general structure analogous
to Figures~\ref{fig:fsa_Dn_odd},~\ref{fig:fsa_Dn_even}, and~\ref{fig:fsa_Dnprime}.
Instead, we describe a simple
expansion rule to derive $\OO_n$ from $\DD_n'$. As
the number of states of $\OO_n$ grows with $n$, we cannot easily
count the number of accepted paths in $\OO_n$ of a given length,
and hence compute its growth
rate using the standard method of analyzing its
transition matrix.  Here, too, we will take advantage of $\DD_n'$
and show how to derive the number of paths in $\OO_n$ of a
given length from the structure of $\DD_n'$.
This is explained in Section~\ref{section:bs1n_growth}.

In order to motivate the expansion of $\DD_n'$ to $\OO_n$, consider
the structure of geodesic paths in $BS(1,n)$ of strict shape $1$
representing $g = t^{-u}a^vt^w$, which have the form
\[
t^{-u} a^{x_0} t a^{x_{1}} \cdots t a^{x_{\kx}}t^{w-\kx}
\]
for a minimal vector $\x \in \Bvuw$.
In other words, the geodesic is composed of an initial power
of $t^{-1}$ followed by an alternating sequence of powers of
$a$ with $t$, where the powers are exactly the digits in  $\x$.  
The fact that $\eta_{v,u,w}(\x)$ has strict shape 1 requires that $\kx < w$ and hence $w- \kx >0$.
As a strict shape $1$ geodesic is allowed to end with an arbitrarily
large power of $t$, we are no longer concerned with making certain
that the digit sequence in question does not end with a $0$ digit.
Any such extraneous $0$ digits correspond in the geodesic to a higher power of $t$,
that is, to $w-\kx$.

To construct a finite state automaton
$\OO_n$ which accepts these geodesics, we must allow
initial strings of $t^{- 1}$'s, require nonempty final strings of $t$'s,
and expand each digit of $\x$ into a sequence of copies
of $a^{\pm 1}$, separated by $t$.

For both even and odd $n$, define the
\emph{$\alpha$-digit expansion} of a state $s_i$
in $\DD_n'$, where $\alpha = \lfloor \frac{n}{2} \rfloor$, to be the collection of states and transitions shown in
Figure~\ref{fig:fsa_digit_expansion}.  Specifically, for each $i$, the state
$s_i$ in $\DD_n'$ is replaced with the collection of states
$$\{s_{i,-\alpha}, s_{i,-\alpha+1}, \ldots, s_{i,0}, \ldots, s_{i,\alpha-1},s_{i,\alpha}\}$$
connected as follows.
\begin{enumerate}[itemsep=5pt]
\item For all $0\le j < \alpha$, there is an edge labeled $a$ from
$s_{i,j}$ to $s_{i,j+1}$.
\item For all $- \alpha < j \le 0$, there is an edge labeled
$a^{-1}$ from $s_{i,j}$ to $s_{i,j-1}$.
\item All edges outgoing from $s_i$ to another state $s_j$
are replaced by edges outgoing from $s_{i,\ell}$ to 
the state $s_{j,0}$ in the $\alpha$-digit expansion of $s_j$, with label $t$.
\end{enumerate}

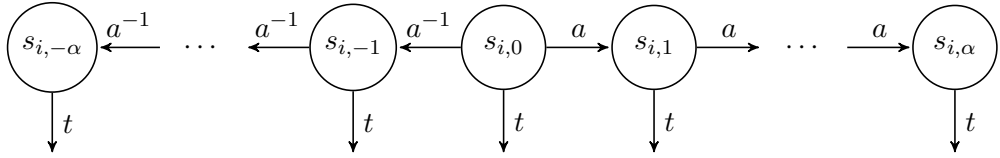
\begin{figure}[ht!]
\tikzset{every state/.style={minimum size=2.9em}} 
\begin{tikzpicture}[->,>=stealth',shorten >=1pt,auto, node distance=2cm,
semithick]

\node[state] (A) {$s_{i,-\alpha}$};
\node[state, draw=none] (B) [right of=A] {$\cdots$};
\node[state] (C) [right of=B] {$s_{i,-1}$};
\node[state] (D) [right of=C] {$s_{i,0}$};
\node[state] (E) [right of=D] {$s_{i,1}$};
\node[state, draw=none] (F) [right of=E] {$\cdots$};
\node[state] (G) [right of=F] {$s_{i,\alpha}$};

\node[state, draw=none] (A1) [below of = A]{};
\node[state, draw=none] (C1) [below of = C]{};
\node[state, draw=none] (D1) [below of = D]{};
\node[state, draw=none] (E1) [below of = E]{};
\node[state, draw=none] (G1) [below of = G]{};

\path(A)
edge node{$t$}(A1);

\path(B)
edge node[above]{$a^{-1}$}(A);

\path(C)
edge node{$t$}(C1)
edge node[above]{$a^{-1}$}(B);

\path(D)
edge node[above]{$a^{-1}$}(C)
edge node{$t$}(D1)
edge node{$a$}(E);

\path(E)
edge node{$t$}(E1)
edge node{$a$}(F);

\path(F)
edge node{$a$}(G);

\path(G)
edge node{$t$}(G1);

\end{tikzpicture}
\vspace{-1.25cm}
\caption{The $\alpha$-digit expansion of a state $s_i$ in $\DD_n'$.}
\label{fig:fsa_digit_expansion}
\end{figure}

We define $\OO_n$ to be the finite state automaton which is obtained
by performing the $\alpha$-digit expansion on every state
in $\DD_n'$ and prepending states $\texttt{start}$ and $s_{t^{-1}}$
in order to allow for any number of initial $t^{-1}$ letters.  It is unnecessary
to append a special state accepting the final sequence of powers of $t$;
such sequences will be ``interpreted'' by $\OO_n$ as a sequence
of digits consisting only of zeros and accepted.
Since an accepted string must end with the letter $t$, we designate the
states $s_{j,0}$ for $j \in \{0,1,2\}$ as accept states.
Except for the start state, these are exactly the states with an incoming edge labeled $t$.

\begin{theorem}\label{thm:fsa_shape_1}
The finite state automaton $\OO_n$ accepts exactly the
language of geodesic paths of strict shape $1$ in $BS(1,n)$.
\end{theorem}
\begin{proof}
This is an immediate consequence of Lemma~\ref{lemma:lex_min_reg}
together with the observations above deriving the language
of strict shape $1$ geodesics from the language of reduced paths
accepted by $\DD_n$.
\end{proof}

\subsection{Example automata for $n=2$}

To illustrate the derivation of $\OO_n$ from $\DD_n'$, we construct these
automata when $n=2$.  Figure~\ref{fig:fsa_2} illustrates
$\DD_2$ and $\DD'_2$, while Figure~\ref{fig:fsa_2_shape_1} illustrates $\OO_2$, 
the result of performing the digit expansion on each state in $\DD_2'$.

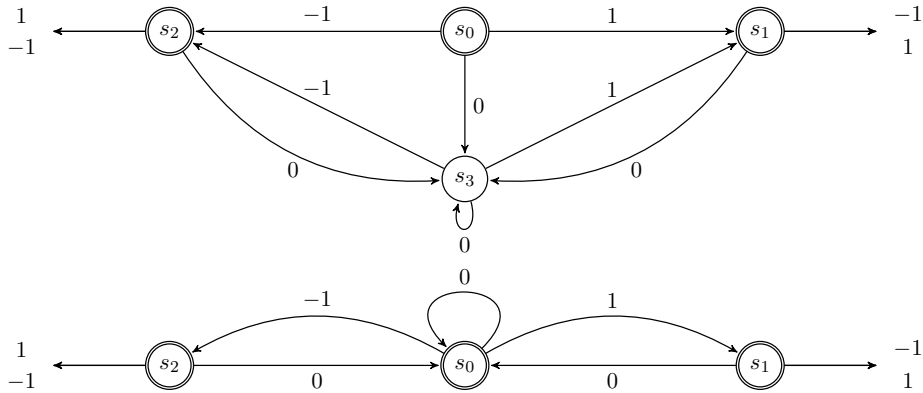
\begin{figure}[ht!]
\tikzset{every state/.style={minimum size=2em}} 
\begin{center}
\resizebox {\columnwidth} {!} {
\begin{tikzpicture}[->,>=stealth',shorten >=1pt,auto, node distance=5cm,
semithick]

\node[state, accepting] (A) {$s_0$};
\node[state, accepting] (B) [right of =A] {$s_1$};
\node[state, draw=none] (BB) [right of=B] {};
\node[state, accepting] (C) [left of = A] {$s_2$};
\node[state, draw=none] (CC) [left of=C] {};
\node[state] (D) [below of = A, yshift=2.5cm] {$s_3$};

\path(A)
edge node[above]{$1$} (B)
edge node[above]{$-1$} (C)
edge node[right]{$0$} (D);

\path(B)
edge [bend left] node[below]{$0$} (D)
edge [shorten >=75pt] node[above] {$-1$} (BB)
edge [shorten >=75pt] node[below] {$1$} (BB);

\path(C)
edge [bend right] node[below] {$0$} (D)
edge [shorten >=75pt] node[above] {$1$} (CC)
edge [shorten >=75pt] node[below] {$-1$} (CC);

\path(D)
edge node[above] {$1$} (B)
edge node[above] {$-1$} (C)
edge [loop below] node[below] {$0$} (D);

\end{tikzpicture}
}

\resizebox {\columnwidth} {!} {
\begin{tikzpicture}[->,>=stealth',shorten >=1pt,auto, node distance=5cm,
semithick]

\node[state, accepting] (A) {$s_0$};
\node[state, accepting] (B) [right of =A] {$s_1$};
\node[state, draw=none] (BB) [right of=B] {};
\node[state, accepting] (C) [left of = A] {$s_2$};
\node[state, draw=none] (CC) [left of=C] {};

\path(A)
edge [loop] node[above] {$0$} (A)
edge [bend left] node[above]{$1$} (B)
edge [bend right] node[above]{$-1$} (C);

\path(B)
edge node[below]{$0$} (A)
edge [shorten >=75pt] node[above] {$-1$} (BB)
edge [shorten >=75pt] node[below] {$1$} (BB);

\path(C)
edge node[below] {$0$} (A)
edge [shorten >=75pt] node[above] {$1$} (CC)
edge [shorten >=75pt] node[below] {$-1$} (CC);

\end{tikzpicture}
}
\end{center}
\caption{The finite state automata $\DD_2$, top, and $\DD_2'$, bottom, for $n=2$.
For both automata, the states $s_0, s_1, s_2$ are accept
states, and $s_0$ is the start state.}
\label{fig:fsa_2}
\end{figure}

\begin{figure}[ht!]
\tikzset{every state/.style={minimum size=2em}} 
\begin{center}
\resizebox {\columnwidth} {!} {
\begin{tikzpicture}[->,>=stealth',shorten >=1pt,auto, node distance=5cm,
semithick]

\node[state,accepting] (A) {$s_{0,0}$};
\node[state] (A1) [above right of =A, yshift=-0.5cm] {$s_{0,1}$};
\node[state] (A_1) [above left of =A, yshift=-0.5cm] {$s_{0,-1}$};

\node[state] (T) [above of =A] {$s_{t^{-1}}$};
\node[state] (start) [above of =T,yshift=-1.5cm] {$\texttt{start}$};

\node[state,accepting] (B) [right of =A] {$s_{1,0}$};
\node[state, draw=none] (BB) [right of=B] {};

\node[state,accepting] (C) [left of = A] {$s_{2,0}$};
\node[state, draw=none] (CC) [left of=C] {};

\path(start)
edge node[left] {$t^{-1}$} (T)
edge [bend left] node[right, pos=0.25] {$t$} (A)
edge [bend left] node[left] {$a$} (A1)
edge [bend right] node[right] {$a^{-1}$} (A_1);

\path(T)
edge [loop below] node[below] {$t^{-1}$} (T)
edge node[above] {$a$} (A1)
edge node[above] {$a^{-1}$} (A_1);

\path(A)
edge [loop below] node[below] {$t$} (A)
edge node[right]{$a$} (A1)
edge node[right]{$a^{-1}$} (A_1);

\path(A1)
edge node[above] {$t$} (B);

\path(A_1)
edge node[above] {$t$} (C);

\path(B)
edge node[above]{$t$} (A)
edge [shorten >=75pt, dashed] node[above] {$a$} (BB)
edge [shorten >=75pt, dashed] node[below] {$a^{-1}$} (BB);

\path(C)
edge node[above] {$t$} (A)
edge [shorten >=75pt, dashed] node[above] {$a$} (CC)
edge [shorten >=75pt,dashed] node[below] {$a^{-1}$} (CC);

\end{tikzpicture}
}
\end{center}
\caption{The finite state automaton $\OO_2$
derived from $\DD_2'$ by expanding each state and prepending
states to allow an initial sequence of $t^{-1}$ letters.
Accept states are indicated with a double circle.  Each accepted string
corresponds to an infinite family of geodesics of strict shape
$1$ in $BS(1,2)$.  We have omitted the expanded states
$s_{1,\pm 1}$ and $s_{2,\pm 1}$ which are unreachable
and have rearranged the states for clarity.}
\label{fig:fsa_2_shape_1}
\end{figure}
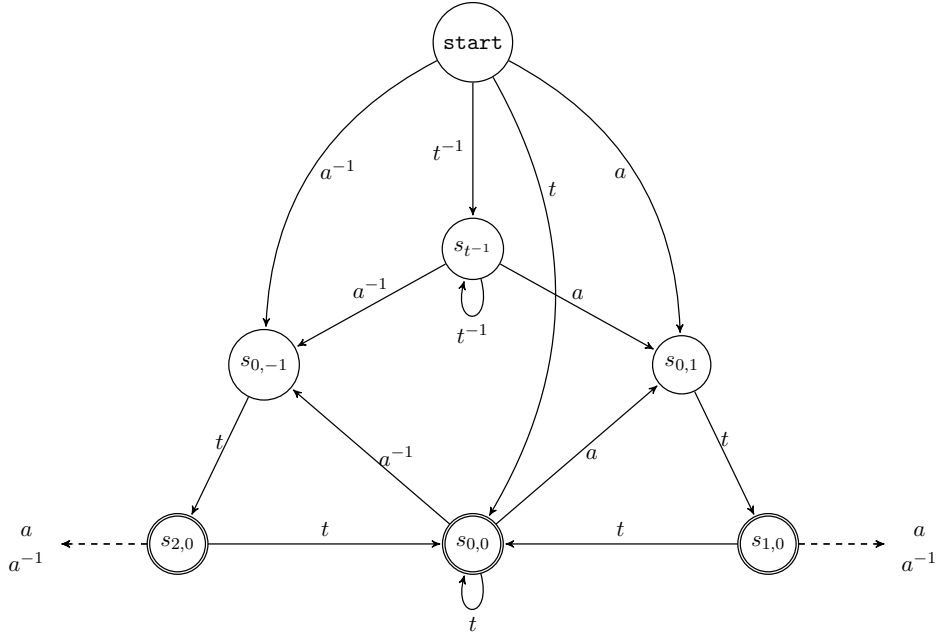

\section{Exponential growth}
\label{section:exp_growth}

In this section we present two lemmas about growth rates which are frequently referenced in Section~\ref{section:bs1n_growth}.
Recall that the growth rate of a sequence $\{f(N)\}_{N=1}^\infty$ is $\lambda$ if and only if
\[
\lim_{N\to\infty} \frac{\log f(N)}{N\log\lambda} = 1.
\]
Equivalently, we write $f(N) = \Theta(\lambda^N)$; that is,
there are constants $A,B>0$ such that $$A \lambda^N \le f(N) \le B \lambda^N$$
for sufficiently large $N$.

\begin{lemma}\label{lemma:exp_growth}
Suppose that $f(N) = \Theta(\lambda^N)$ with $\lambda > 1$.
\begin{enumerate}[itemsep=5pt]
\item Both $f(N+k)$ and $\sum_{i=1}^Nf(i)$ are $\Theta(\lambda^N)$.
\item If $f(N)$ and $g(N)$ are $ \Theta(\lambda^N)$, there are $N_0,d>0$ so that $f(N)/g(N) > d$ for $N > N_0$.
\end{enumerate}
\end{lemma}
\begin{proof}
It is clear that $f(N+k) = \Theta(\lambda^{N+k}) = \Theta(\lambda^N)$.  We now
show that $\sum_{i=1}^Nf(i)=\Theta(\lambda^N)$.  As discussed above, there are
$C_1,C_2,M>0$ so that for all $N>M$ we have
\[
C_1 \lambda^N \le f(N) \le C_2 \lambda^N.
\]
Let $D= \sum_{i=1}^{M}f(i)$; note that $D$ is  constant.  The inequalities
\[
C_1 \sum_{i=M+1}^N \lambda^i + D \le \sum_{i=1}^N f(i) \le C_2  \sum_{i=M+1}^N \lambda^i + D,
\]
combined with the expansion $\sum_{i=1}^N\lambda^i = \lambda(\lambda^{N}-1)/(\lambda-1) = \Theta(\lambda^N)$, yield additional
constants $C_3,C_4>0$ so that
\[
C_3 \lambda^N \le C_1 \sum_{i=M+1}^N \lambda^i + D \le \sum_{i=1}^N f(i) \le C_2  \sum_{i=M+1}^N \lambda^i + D \le C_4 \lambda^N
\]
for sufficiently large $N$.
Thus $\sum_{i=1}^Nf(i)=\Theta(\lambda^N)$, as desired.  

To prove
the second statement in the lemma, observe that for sufficiently
large $N$ and new constants $C_i>0$ we have
\[
C_1 \lambda^N \le f(N) \le C_2 \lambda^N \text{ and }
C_3 \lambda^N \le g(N) \le C_4 \lambda^N,
\]
hence $f(N)/g(N) > C_1/C_4 > 0$.
\end{proof}

We will be interested in determining the growth rate of the
function which counts the number of accepted paths of a given length
in a finite state automaton.

\begin{lemma}\label{lemma:fsa_growth}
Let $F$ be a finite state automaton with state set $S$.
Let $f(N)$ denote the number of accepted paths in $F$ of length $N$,
and for each $s \in S$, let $f_s(N)$ denote the number of accepted
paths in $F$ beginning at state $s$.
Let $F_1, \dots, F_c$ be the strongly connected components of $F$.
\begin{enumerate}
    \item For each $i$, the growth rate of $\{f_s(N)\}_{N \in \N}$ is constant over all $s \in F_i$.
    \item The growth rate of $\{f(N)\}_{N \in \N}$ is the maximum of the growth rates of the $F_i$.
\end{enumerate}
\end{lemma}
\begin{proof}
Let $\lambda$ be the growth rate of the sequence $\{f(N)\}_{N \in \N}$, and $\lambda_s$ the 
growth rate of the sequence $\{f_s(N)\}_{N \in \N}$ for any $s \in S$.
Let $s, s' \in S$ be states.  
If there is an edge from $s$ to $s'$, 
then $f_s(N+1) \ge f_{s'}(N)$.  
It follows from Lemma~\ref{lemma:exp_growth} that $\lambda_s \ge \lambda_{s'}$.
Iterating this argument shows
that if there is a path of any length from $s$ to $s'$, then
$\lambda_s \ge \lambda_{s'}$ and hence $\lambda_s$ is constant
over a strongly connected component in $F$.

Next observe that if $s$ is any state in $F$, then
$f_s(N) = \sum_{s\mapsto s'} f_{s'}(N-1)$, where $s \mapsto s'$
denotes the existence of an edge from $s$ to $s'$.  Therefore,
$\lambda_s = \max_{s\mapsto s'}\lambda_{s'}$.  Iterating
this argument shows that $\lambda_s$ is the maximum $\lambda_{\bar{s}}$
over all $\bar{s}$ which are reachable from $s$.  Applying this
argument to the start state proves the second statement of the lemma.
\end{proof}

\section{The growth rate of $BS(1,n)$}
\label{section:bs1n_growth}

Let $S_n(N)$ denote the sphere
of radius $N$ in $BS(1,n)$.  The growth rate of a finitely generated group $G$
is defined to be the growth rate of the sequence $\{|S_n(N)|\}_{n \in \N}$.
In this section we compute the growth rate of $BS(1,n)$ for all $n>1$ using the finite state automaton constructed in 
Section~\ref{section:regular_languages} which accepts $\OO_n$, the language of
geodesic paths of strict shape $1$.  

In order to obtain bounds
on the number of minimal vectors producing geodesic paths of any shape, we construct a map from the set of all minimal
vectors to the set of minimal vectors corresponding to geodesic paths of strict shape $1$. 
The difficulty is that there are certain digits allowed in a minimal vector $\x$ which are not permitted as
exponents in a geodesic path of strict shape 1. Specifically, when
$\kx \geq \max(u,w)$, the final digit of $\x$ is allowed to have absolute value greater than $\fn$, but a minimal vector corresponding to a geodesic
of strict shape 1 must satisfy $\kx < \max(u,w)$ and have $|x_i| \leq \fne$ for all $i\leq \kx < \max(u,w)$. 

The idea of our map is not complicated:
given $\x\in\Bvuw$, we assume that $\kx < \max(u,w)$, possibly
reducing $\x$  to remove a final
digit whose absolute value exceeds $\fn$, 
and modify $u$ and/or $w$ to $u'$ and/or $w'$, ensuring that $\kx < \max(u',v')$.

We now define a map $c:\LL_v \to \LL_v$ which implements the algorithm described above.  The subsequent modification of $u$ and/or $w$ is a secondary step.
Take $\x \in \Bvuw$.  
If it exists, let $j\le \kx$ be the minimal index such that $|x_i| \ge \fne$ for $j\le i \le \kx$.

It is easy to verify that for $j\le i < \kx$, the digits $x_i$ have constant sign. Suppose $(x_i,x_{i+1}) = (\delta\fne,-\delta\fne)$.
Let $\y = \x + \delta\wi$.  Then $|\kx-\ky| \leq 1$ which implies that we can use the same length formula from Lemma~\ref{lemma:length_formula} to determine whether 
$|\eta_{u',v',w'}(\y)| < |\eta_{u,v,w}(\x)|$.
It is easy to verify that $\Vert \y \Vert_1 < \Vert \x \Vert_1$, and thus, regardless of which length formula is required, $|\eta_{u',v',w'}(\y)| < |\eta_{u,v,w}(\x)|$.
Thus $\x$ can be reduced, contradicting the fact that it is minimal.

Note that the digits of $c(\x)$ and $\x$ are identical for indices less than $j$.

When $n$ is odd it follows from by Lemma~\ref{lemma:odd_box} that the only possibility we must consider is $j = \kx$ with $x_{\kx} = \delta(\fn +1)$, where $\delta \in \{\pm 1\}$.
When $n>2$ is even, $(x_j, \cdots ,x_{\kx})$ is a maximal
sequence where all but the final digit is $\delta\fne$, and the final digit might be $\delta(\fne+1)$.
When $n=2$, the final digit is chosen from the set $\{\delta,\delta 2,\delta 3\}$.
Note that this sequence of digits is likely quite short as the number of repetitions of the digit $\fn$ is always limited, as shown in \cite{TW}.

The map $c$ requires an alternate definition when $n=2$ to account for the differing bounds on the final digit of a vector in $\Bvuw$.
Let ``condition (A)" denote the case $n=2$ and either
\begin{itemize}[itemsep=5pt]
    \item $x_\kx = \delta 3$ or
    \item $x_\kx=\delta 2$ and $j < \kx$.
\end{itemize}
If condition (A) holds, define
\[
c(\x) = \x + \delta\sum_{i=j}^{\kx-1} \wi + \delta 2\wf{\kx} + \delta \wf{\kx+1}.
\]
Otherwise, if $x_{\kx} = \delta(\fn+1)$ or $j < \kx$, that is, the last digit
has absolute value greater than $\fn$ or the sequence has length at least 2, define
\[
c(\x) = \x + \delta\sum_{i=j}^{\kx} \wi.
\]
When $j$ is undefined or there is a single $\delta\fne$ at the end of $\x$,
define $c(\x) = \x$.

When $n$ is odd, $c(\x)$ replaces a final digit in $\x$ of $\delta(\fn+1)$
with the final digits $(-\delta(\fn-1),\delta)$.
When $n$ is even, the behavior of $c(\x)$ depends on the configuration
of the final digits of $\x$.  This behavior is straightforward, but we will
need to refer to this computation later, so we explain it in detail.
The simplest case is that there is a sequence of at least two $\delta\fne$ digits
at the end of $\x$.  We compare the digits of $\x$ and $c(\x)$.
\begin{equation}\label{equation:cx_digits}
\begin{array}{cccccrcrcr}
    \x    & =& (x_0, & \cdots, & x_{j-1}, &  \delta\fne, & \delta\fne, & \cdots, & \delta\fne) \\
    c(\x) & =& (x_0, & \cdots, & x_{j-1}, & -\delta\fne, & -\delta(\fne-1),& \cdots,& -\delta(\fne-1),& 1)
\end{array}
\end{equation}
We now describe the slight variations in the other cases.
\renewcommand{\labelitemii}{$\circ$}
\begin{itemize}[itemsep=5pt]
    \item When subsequence has length at least $2$ and the final digit of $\x$ is $\delta(\fne+1)$,
    then the penultimate digit of $c(\x)$ is $-\delta(\fne-2)$ and all other digits are as in Equation~\eqref{equation:cx_digits}.
    \item When there are no $\delta\fne$ digits, a final digit of $\delta(\fne+1)$ is replaced by the sequence of digits $(-\delta(\fne-1),1)$.
    \item When $n=2$, we use the symbol $|$ to mark a location in the vector so that the change in digits is clearly depicted.
    When $\x$ ends with the maximal subsequence:
    \begin{itemize}[itemsep=5pt]
        \item $(\delta,\delta,\dots,\delta,|\delta 3)$, it is replaced with
    $(-\delta,0,\dots,0,|0,0,\delta)$, where the $\delta 3$ is replaced by $(0,0,\delta)$.
        \item $(\delta 3)$, it is replaced with $(-\delta,0,\delta)$.
        \item $(\delta,\delta,\dots,\delta,|\delta 2)$, it is replaced with
    $(-\delta,0,\dots,0,|-\delta,0,\delta)$.
        \item $(\delta 2)$, it is replaced with $(0,\delta)$. 
    \end{itemize}
\end{itemize}

A computation shows that when condition (A) holds
we have $k_{c(\x)} = \kx + 2$.  
In all other cases, we have $k_{c(\x)} = \kx+1$.
The change in $\ell^1$ norm depends on the length $\kx-j+1$ of the digit sequence $(x_j, \cdots ,x_{\kx})$ and the
value of the final digit.  Specifically, we observe that
\begin{equation}\label{equation:cx_norm_change}
\Vert c(\x) \Vert_1 = \left\{
\begin{array}{ll}
\Vert \x \Vert_1               & \textnormal{if $n$ is odd or $c(\x) = \x$} \\
\Vert \x \Vert_1 - (\kx - j-1) & \textnormal{if $|x_\kx| = \fne$} \\
\Vert \x \Vert_1 - (\kx - j-1) & \textnormal{if $n=2$ with $|x_{\kx}|=2$ and $j<\kx$} \\
\Vert \x \Vert_1 - (\kx - j+1) & \textnormal{otherwise.}
\end{array}\right.
\end{equation}
\begin{lemma}\label{lemma:cx_at_most_x}
Let $\x \in \Bvuw$ and define the map $c$ as above.  Then
 $\Vert c(\x) \Vert_1 \le \Vert \x \Vert_1$.
\end{lemma}
\begin{proof}
When $c(\x) = \x$ or $\kx-j \ge 1$, the lemma is immediate.  
If $\kx =j$, it follows from the definition of $c(\x)$ that $|x_\kx| > \fne$, so we are in the final
case of Equation~\ref{equation:cx_norm_change} above.  Here $\kx -j + 1 > 0$, so
the lemma follows.
\end{proof}

In order to compare geodesic length before and after our alteration of the vector $\x$, define
\[
\beta(\x) = 3 +  \Vert \x \Vert_1 - \Vert c(\x) \Vert_1.
\]
It follows immediately from Lemma~\ref{lemma:cx_at_most_x} that $\beta(\x) \ge 3$.

Define
\[
\Phi(u,w,\x) = \left\{\begin{array}{ll}
(u,w+\beta(\x),c(\x)) & \textnormal{if $\eta_{u,v,w}(\x)$ has shape $1$} \\
(w,u+\beta(\x),c(\x)) & \textnormal{if $\eta_{u,v,w}(\x)$ has shape $2$} \\
(u,2\kx-w+\beta(\x),c(\x)) & \textnormal{if $\eta_{u,v,w}(\x)$ has shape $3$} \\
(w,2\kx-u+\beta(\x),c(\x)) & \textnormal{if $\eta_{u,v,w}(\x)$ has shape $4$} 
\end{array}
\right.
\]

\begin{lemma}\label{lemma:shape_map}
Let $\x \in \Bvuw$.  If $(u',w',c(\x)) = \Phi(u,w,\x)$, then 
\[
|\eta_{u',v',w'}(c(\x))| = |\eta_{u,v,w}(\x)|+3,
\]
where $v'$ is determined by $c(\x)$
and $\eta_{u',v',w'}(c(\x))$ has strict shape $1$. Moreover, if
$\x$ is minimal, then $c(\x) \in \BB_{v'}^{u',w'}$ is also minimal.
\end{lemma}
\begin{proof}
We begin with the assumption that $\eta_{u',v',w'}(c(\x))$ has strict shape $1$ and first prove that $|\eta_{u',v',w'}(c(\x))| - |\eta_{u,v,w}(\x)|=3$.

Note that if $\eta$ is a geodesic of strict shape 1, we use the first length formula in Lemma~\ref{lemma:length_formula} to compute its length.
It follows that
$|\eta_{u',v',w'}(c(\x))| = \Vert \x' \Vert_1 + u' + w'$.
When $\eta_{u,v,w}(\x)$ has shape 1 or 2, we again use the first length
formula in Lemma~\ref{lemma:length_formula} to compute its length.
Recalling the definition of $\beta$, we compute
\begin{align*}
|\eta_{u',v',w'}(c(\x))| &= \Vert c(\x) \Vert_1 + u' + w' \\
                       &= \Vert c(\x) \Vert_1 + u + w + \beta(\x) \\
                       &= \Vert \x \Vert_1 + u + w + 3 \\
                       &= |\eta_{u,v,w}(\x)|+3.
\end{align*}
When $\eta_{u,v,w}(\x)$ has shape 3 we must have $u  < \kx$, and we use the second formula in Lemma~\ref{lemma:length_formula} to compute $|\eta_{u,v,w}(\x)| = \Vert \x \Vert_1 + 2\kx -|u-w|$.
It follows that
\begin{align*}
|\eta_{u',v',w'}(c(\x))| &= \Vert c(\x) \Vert_1 + u' + w' \\
    &= \Vert c(\x) \Vert_1 + u + 2\kx - w + \beta(\x) \\
      &= \Vert \x \Vert_1 +3 + 2\kx - |u-w| \\
      &= |\eta_{u,v,w}(\x)|+3.
\end{align*}
An analogous computation yields the same conclusion when $\eta_{u,v,w}(\x)$ has shape 4.
Thus it remains to
show that $\eta_{u',v',w'}(c(\x))$ has strict shape $1$ and $c(\x) \in \BB_{v'}^{u',w'}$
is minimal, that is, $\eta_{u',v',w'}(c(\x))$ is a geodesic.

We begin by proving that $\eta_{u',v',w'}(c(\x))$ has strict shape $1$.
Suppose that $\eta_{u,v,w}(\x)$ has shape 1 or 2.  From the definition of $c(\x)$ we know that $k_{c(\x)} \leq \kx +2$.  If $\eta_{u,v,w}(\x)$ has shape 1, then $\kx \leq w$ and hence
\[
k_{c(\x)} \leq \kx +2 < \kx + 3 \leq w+3 \leq w + \beta(\x) = w'.
\]
That is, $c(\x)$ produces a geodesic of strict shape 1. An analogous argument holds when $\eta_{u,v,w}(\x)$ has shape 2.

When $\eta_{u,v,w}(\x)$ has shape $3$ we have $w < \kx$,
so $2\kx - w > \kx$.
To verify that $\eta_{u',v',w'}(c(\x))$ has strict shape 1
we must check that $k_{c(\x)} < w'$.
Since $\beta(\x) \ge 3$ and $k_{c(\x)} \leq \kx +2$, it follows that $k_{c(\x)} < \kx + \beta(\x)$ and
\[
k_{c(\x)} < \kx + \beta(\x) < 2\kx -w + \beta(\x) = w'.
\]
A similar relation holds when $\eta_{u,v,w}(\x)$ has shape $4$.

By construction,
$c(\x)$ satisfies the digit bounds on $\BB_{v'}^{u',w'}$.
It remains to show that $c(\x) \in \BB_{v'}^{u',w'}$ is minimal in all cases.
When $n$ is odd, the fact that $k_{c(\x)} < \max(u',w')$ together with Lemma~\ref{lemma:odd_box}
imply that $c(\x)$ is minimal.  

Next let $n$ be even.
Note that for $i<j$ we have $|c(\x)_{i}| = |x_{i}|$; additionally we have $|c(\x)_{j-1}| = |x_{j-1}| < \fne$. 
Suppose towards a contradiction that $c(\x)$ is not minimal.
It follows from  Lemma~\ref{lemma:adjacent_digits} that
$c(\x)$ contains a digit subsequence $(c(\x)_{i},c(\x)_{i+1}) = (\delta\fne,\delta\fne)$ or $(c(\x)_{i},c(\x)_{i+1})=(\delta\fne,c(\x)_{i+1})$
where $\sgn(c(\x)_{i+1}) = -\sgn(\delta)$.  

Suppose $c(\x)$ contains the digit subsequence $(c(\x)_i,c(\x)_{i+1}) = (\delta\fne,\delta\fne)$.  
The definition of $c(\x)$ precludes any digit $|c(\x)_m| \ge \fne$
for any $m$ with $j+1 \le m < k_{c(\x)}$ and  $|c(\x)_{j-1}| < \fne$.  Therefore, the indices
of both digits in the subsequence must be strictly less than $j-1$.  
Thus $(c(\x)_i,c(\x)_{i+1})= (x_i,x_{i+1})$ is contained in $\x$.
Including the subsequent digit, we have $(x_i,x_{i+1},x_{i+2}) = (\delta\fne,\delta\fne,x_{i+2})$, and $|x_{i+2}| < \fne$.

Let $\y = \x + \delta(\wf{i-2}+\wf{i-1})$.
It is easily verified that $\Vert \y \Vert_1 \leq \Vert \x \Vert_1$, as $(x_i,x_{i+1},x_{i+2}) = (\delta\fne,\delta\fne,x_{i+2})$ and $(y_i,y_{i+1},y_{i+2}) = (-\delta\fne,-\delta(\fne-1),x_{i+2}+\delta)$.
If there is equality between the two $l^{1}$ norms, note that the change from $x_{i+1}$ to $y_{i+1}$ is a lexicographic decrease.

Suppose $c(\x)$ contains the digit subsequence $(c(\x)_i,c(\x)_{i+1}) = (\delta\fne,c(\x)_s)$
where $\sgn(c(\x)_s) = -\sgn(\delta)$.  
It follows that $i+1 \leq j-1$ and thus $(c(\x)_i,c(\x)_{i+1}) = (x_i,x_{i+1})$.
Let $\y = \x + \wi$.  It is easily verified that $\Vert \y \Vert_1 < \Vert \x \Vert_1$.

In both cases, as $|\kx-\ky| \leq 1$ we can use the same length formula from Lemma~\ref{lemma:length_formula} to determine whether $|\eta_{u',v',w'}(\y)| < |\eta_{u,v,w}(\x)|$.
It follows from the comparison of $\Vert \y \Vert_1$ and $\Vert \x \Vert_1$ that $|\eta_{u',v',w'}(\y)| < |\eta_{u,v,w}(\x)|$., contradicting the fact that $\x$ is minimal.  
Thus it must be the case that $c(\x) \in \BB_{v'}^{u',w'}$ is minimal.
\end{proof}

The next lemma allows us to compute the degree of the map $\Phi$, which will be crucial to the proof of Corollary~\ref{corollary:shapes_injection}, where we show that the growth rates of the sequences $\OOn$ and $\{|S_n(N)|\}_{n \in \N}$ are identical.  
Let $\OO_n(N)$ denote the set of elements of the language $\OO_n$
which are of length $N$.

\begin{lemma}\label{lemma:clamp_degree}
For any minimal vector $\y \in \BB_{v'}^{u',w'}$
the maximal number of minimal vectors $\x\in \Bvuw$ such that $c(\x) = \y$
is 
\begin{itemize}[itemsep=5pt]
    \item $5$ if $n=2$,
    \item $3$ if $n>2$ is even, and
    \item at most $2$ when $n$ is odd.
\end{itemize}
\end{lemma}
\begin{proof}
The proof reduces to considering the digit comparison
in Equation~\eqref{equation:cx_digits} and its variants, and observing
the relationship between the final digits of $c(\x)$ and the final digits of $\x$.  

First let $n>2$ be even.
In the vector subsequences below, any terms in square brackets may be omitted from the expression.
If the vector $\y$ ends with a subsequence $(y_j, \dots ,y_{\ky})$ of the form 
\begin{itemize}[itemsep=5pt]
\item $(-\delta\fne, \left[ -\delta(\fne-1), \dots,-\delta(\fne-1)\right], -\delta(\fne-2), 1)$,
then a preimage under $c$ is $\x = \y - \sum_{i=j}^{\ky-1} \wf{i}$.
\item $(-\delta\fne,  \left[-\delta(\fne-1), \dots,-\delta(\fne-1) \right],-\delta(\fne-1), 1)$, 
then a preimage under $c$ is $\x = \y - \sum_{i=j}^{\ky-1} \wf{i}$.
\item $(-\delta(\fne-1), 1)$, then a preimage under $c$ is $\x = \y -\wf{\ky-1}$.
\end{itemize}
For all $\y$, it might be the case that $\x=\y$ is a preimage. By construction, the preimages listed
above are the only possibilities.  Only two sequences of the above forms
above can overlap.  That is, $\y$ can only contain at most two subsequences of the forms
above.  Therefore, $\y$ can have at most $3$ preimages if $n>2$ is even.
When $n$ is odd, only the final bullet above applies, and we conclude that $\y$ has at most two preimages in this case.

We perform the same analysis when $n=2$.  If the vector $\y$ ends with a subsequence
$(y_j, \dots ,y_{\ky})$ of the form 
\begin{itemize}[itemsep=5pt]
\item $(-\delta,0,\delta)$, then a preimage under $c$ is 
\[\x = \y - 2\delta\wf{\ky-2} - \delta\wf{\ky-1}.\]
\item $(-\delta,[0,\dots,0],0,0,\delta)$, then a preimage under $c$ is
\[\x = \y - \delta(\sum_{i=j}^{\ky-3}\wf{i}+ 2\wf{\ky-2} + \wf{\ky-1}).\]
\item $(-\delta,[0,\dots,0],-\delta,0,\delta)$, then a preimage under $c$ is
\[\x = \y - \delta(\sum_{i=j}^{\ky-3}\wf{i}+ 2\wf{\ky-2} + \wf{\ky-1}).\]
\item $(0,\delta)$, then a preimage under $c$ is $\x =\y - \delta\wf{\ky-1}$.
\item $(-\delta,[0,\dots,0],0,\delta)$, then a preimage under $c$ is $\x = \y - \delta\sum_{i=j}^{\ky-1}\wf{i}$.
\end{itemize}
Again it might be the case that $\x = \y$ is a preimage.  
At most four subsequences of these forms can overlap, so we conclude that $\y$ can have at most five preimages under $c$.
\end{proof}

We now show that the cardinality of the set of all geodesics
is within a uniform constant multiple of the cardinality
of the set of geodesics of strict shape $1$.  
By carefully considering geodesics of different shapes, we could obtain a stronger inequality in the following corollary of Lemmas~\ref{lemma:shape_map}
and~\ref{lemma:clamp_degree}.
However, our result is sufficient to prove Theorem~\ref{theorem:growth_rate}.

\begin{corollary}\label{corollary:shapes_injection}
In the notation above, we have
\[
|\OO_n(N)| \le |S_n(N)| \le 20|\OO_n(N+3)|.
\]
Consequently,
the growth rates of the sequences $\OOn$ and $\{|S_n(N)|\}_{n \in \N}$ are identical.
\end{corollary}
\begin{proof}
Clearly $|\OO_n(N)| \le |S_n(N)|$.
It follows from Lemma~\ref{lemma:shape_map} that  $\Phi$
maps $S_n(N)$ to $\OO_n(N+3)$, 
so to prove the rightmost inequality we must show that
the degree of $\Phi$ is bounded above by $20$.
First we apply Lemma~\ref{lemma:clamp_degree}
to conclude that the degree of $c$ is at most $5$.
As $\Phi$ is defined in
four cases, we see that any output triple
$(u',w',\x')$ could arise from at most five preimages of $\x'$
in each of the four cases of $\Phi$.  Thus the degree of $\Phi$
is at most $20$, as desired.

It follows from Lemma~\ref{lemma:exp_growth} that the growth rates of $\OOn$ and $\{|S_n(N)|\}_{n \in \N}$ are identical.
\end{proof}

Recall that the \emph{growth series} of a function $f(N)$
is an infinite series
$R(x) = \sum_{i=0}^\infty f(i)x^i$.  We will be interested
in growth series which records the number of paths in a
finite automata starting at a given state.

It follows from Corollary~\ref{corollary:shapes_injection} that to determine the
growth rate of $BS(1,n)$ it suffices
to determine the growth rate of the sequence $\OOn$.
As discussed in Section~\ref{section:regular_languages}, it is nontrivial
to do this in general because the number of states in $\OO_n$
is not uniformly bounded.  Our solution is to formally define
a growth series for every state in $\DD_n'$ which counts the number of
paths in $\OO_n$ starting in that state.  We can then write
down a matrix equation of fixed size in these series and
determine their growth rates.  That is, we trade a computation
with arbitrarily large matrices over the integers (computing
an eigenvalue) for a computation with fixed size matrices whose
entries are infinite series.
Also note that for the purpose of computing the growth rate of the sequence $\OOn$, we will ignore
the state $s_{\ti}$; it follows from Lemma~\ref{lemma:fsa_growth} that the growth
rate of the strongly connected component containing
the $\alpha$-digit
expansions of all the states $s_i$ will determine the growth rate of the sequence  $\OOn$.

\begin{theorem}
\label{theorem:growth_rate}
Let $\alpha = \lfloor\frac{n}{2}\rfloor$.  The growth rate of the sequence
$\OOn$, and hence $NS(1,n)$, for $n$ odd is the reciprocal
of the smallest magnitude root of
\[
1-x-\sum_{i=1}^\alpha 2x^{i+1},
\]
and for $n$ even is the reciprocal of
the smallest magnitude root of
\[
1-2x - x^2 + 2x^{\alpha+1} - 2x^{\alpha+2} + 2x^{2\alpha+2}.
\]
\end{theorem}

\begin{proof}
First assume that $n$ is odd.
Denote by $R(x)$ the growth series of paths starting
at state $s_{0,0}$.  By consulting Figure~\ref{fig:fsa_Dnprime},
we see that for each $i$ with $-\alpha \le i \le \alpha$,
there is a path of length $|i|+1$ in $\OO_n$ which returns
to state $s_{0,0}$ consisting
of $|i|$ edges with label $a^{\pm 1}$ followed by $t$.  Thus we can write the recurrence
\[
R(x) = \left[x + 2\sum_{i=1}^\alpha x^{i+1}\right]R(x) + P(x),
\]
where $P(x)$ is a polynomial of degree at most $\alpha+1$.
This ``error'' polynomial arises because not all ways of returning
to state $s_{0,0}$ are possible
if the remaining allowed path is shorter than length $\alpha+1$.
Thus we have
\[
R(x) = P(x)\left[1-x-\sum_{i=1}^\alpha 2x^{i+1}\right]^{-1}.
\]
It is a general fact about growth series that the growth
rate is the reciprocal of the smallest root of the
denominator when the growth series is expressed as a
rational function.  See, for example,~\cite{Flajolet} \S IV.
Solving the recurrence above for $R(x)$
and applying this fact proves the theorem in this case.

Now assume that $n$ is even.  Here $\alpha = \fn = \fne$.
For each state $s_i$ in $\DD_n'$,
denote the growth series of paths starting at the state
$s_{i,0}$ in $\OO_n$ by $R_i(x)$.  We will express each of
the $R_i(x)$ as a polynomial combination of all the
$R_j(x)$.  We can then
solve a matrix equation to find these growth rates.
First consider $R_0(x)$.
By consulting Figure~\ref{fig:fsa_Dnprime}, we see that
for each $i$ with $-\alpha < i < \alpha$, there is a path of
length $|i|+1$ in $\OO_n$  which
returns to state $s_{0,0}$ consisting of $|i|$ edges with label $a^{\pm}$
followed by $t$.  There is also
a path of length $\alpha+1$ to state $s_{1,0}$ and a
path of length $\alpha+1$ to state $s_{2,0}$.  Thus we can write
the recurrence
\[
R_0(x) = \left[x + \sum_{i=1}^{\alpha-1}2x^{i+1}\right]R_0(x)
       + x^{\alpha+1}R_1(x) + x^{\alpha+1}R_2(x) + P_0(x),
\]
where $P_0(x)$ is a polynomial of degree at most $\alpha+1$.
This ``error'' polynomial arises because not all ways of returning
to $s_{0,0}$ or transiting to the other states are possible
if the remaining allowed path is shorter than length $\alpha+1$.

If we start in state $s_{1,0}$, then for each $i$
with $0 \le i < \alpha$,
there is a path of length $|i|+1$ to $s_{0,0}$.  Therefore,
\[
R_1(x) = \left[\sum_{i=0}^{\alpha-1} x^{i+1}\right]R_0(x) + P_1(x).
\]
The computation for $R_2(x)$ is similar.  

These computations yield the
matrix equation
{\small
\[
\left[
\begin{array}{c}
R_0(x) \\
R_1(x) \\
R_2(x) \\
\end{array}
\right]
=
\left[
\begin{array}{ccc}
x + 2\sum_{i=1}^{\alpha-1} x^{i+1} & x^{\alpha+1} & x^{\alpha+1} \\
\sum_{i=0}^{\alpha-1} x^{i+1}& 0            & 0            \\
\sum_{i=0}^{\alpha-1} x^{i+1}& 0            & 0            \\
\end{array}
\right]
\left[
\begin{array}{c}
R_0(x) \\
R_1(x) \\
R_2(x) \\
\end{array}
\right]
+ \left[
\begin{array}{c}
P_0(x) \\
P_1(x) \\
P_2(x) \\
\end{array}
\right].
\]
}
Letting $R(x)$ be the column vector of growth series, $P(x)$ the column vector of polynomials of degree at most $\alpha+1$, and $M(x)$
the matrix, we wish to solve the equation $(I-M(x))R(x) = P(x)$.
Cramer's rule implies that
$\det(I-M(x))$ is the denominator for each generating function
and thus the reciprocal of its smallest root is
the growth rate.  A computation to simplify $\det(I-M(x))$
yields the polynomial stated in the theorem.

It follows immediately from Corollary~\ref{corollary:shapes_injection} that the growth rate of $BS(1,n)$ is also given by the reciprocal of the root of smallest magnitude of the above polynomials.
\end{proof}

\subsection{Growth rate examples and the limiting case}

Using Theorem~\ref{theorem:growth_rate}, it is straightforward
to compute the growth rate of $BS(1,n)$ for small values
of $n$.  These are shown in Table~\ref{fig:growth_rates}.

\begin{table}[ht!]
\begin{center}
\begin{tabular}{ccc}
$n$ & smallest root & reciprocal (growth rate) \\
\hline
2 & 0.589754512301458  &   1.69562076955986 \\
3 &  1/2               &         2 \\
4 & 0.456552637014853  &   2.19032794671486 \\
5 & 0.440619700538199  &   2.26953084208114 \\
6 & 0.428577480668369  &   2.33330038349307 \\
7 & 0.423853799069783  &   2.35930408597178 \\
8 & 0.417979169653687  &   2.39246372212410 \\
$\vdots$ & $\vdots$  & $\vdots$ \\
$\infty$ & $\sqrt{2}-1$ & $\sqrt{2} + 1$
\end{tabular}
\end{center}
\caption{The growth rate of $BS(1,n)$ for small values of $n$.  The meaning
of the limiting growth rate for $n=\infty$ is discussed
in the text.}
\label{fig:growth_rates}
\end{table}

Table~\ref{fig:growth_rates} lists a growth rate for $n=\infty$, indicating the limit of the growth rates for $BS(1,n)$ as
$n\to\infty$.  We can compute this quantity both independently and as a double-check on the polynomials in
Theorem~\ref{theorem:growth_rate}.  In the odd case, note that on any open disk
of radius less than $1$, as $n\to\infty$ the sequence of
polynomials  in
Theorem~\ref{theorem:growth_rate} converges uniformly to the power series
\[
1-x-2\sum_{i=1}^\infty x^{i+1} = 1-x-\frac{2x^2}{1-x},
\]
whose smallest root is $\sqrt{2}-1$.  In the even case, the
sequence of polynomials converges uniformly to $1-2x-x^2$,
which has the same roots as the power series above,
so the even and odd cases agree in the limit.

As an independent check on the limiting case, consider what
a geodesic of strict shape $1$ would look like for ``infinite" $n$:
there would be no bound on the powers of $a$ and $a^{-1}$,
so the set of geodesics of strict shape $1$ would be a regular language
on the three symbols $\{a,a^{-1},t\}$, subject to the condition that $a$ and $\ai$ are never adjacent.
We ignore here the
possible initial power of $t^{-1}$, which does not affect the
growth rate.  A finite automata accepting this language
has adjacency matrix
\[
\left[ \begin{array}{ccc}
 1 & 1 & 1 \\
 1 & 1 & 0 \\
 1 & 0 & 1
 \end{array}\right]
\]
whose largest eigenvalue is $\sqrt{2}+1$.

\bibliographystyle{plain}
\bibliography{main}

\begin{thebibliography}{1}

\bibitem{BDK}
A.~{Bar-Natan}, M.~{Duchin}, and R.~{Kropholler}.
\newblock {Conjugation curvature for Cayley graphs}.
\newblock {\em Journal of Topology and Analysis}, to appear.

\bibitem{Brazil}
Marcus Brazil.
\newblock Growth functions for some nonautomatic {B}aumslag-{S}olitar groups.
\newblock {\em Trans. Amer. Math. Soc.}, 342(1):137--154, 1994.

\bibitem{BT}
Michelle Bucher and Alexey Talambutsa.
\newblock Minimal exponential growth rates of metabelian {B}aumslag-{S}olitar
  groups and lamplighter groups.
\newblock {\em Groups Geom. Dyn.}, 11(1):189--209, 2017.

\bibitem{CEH}
Laura Ciobanu, Alex Evetts, and Meng-Che~"Turbo" Ho.
\newblock The conjugacy growth of the soluble {B}aumslag-{S}olitar groups,
  2019.

\bibitem{CEG}
D.~J. Collins, M.~Edjvet, and C.~P. Gill.
\newblock Growth series for the group {$\langle x,y|\ x^{-1}yx=y^l\rangle$}.
\newblock {\em Arch. Math. (Basel)}, 62(1):1--11, 1994.

\bibitem{EH}
Murray Elder and Susan Hermiller.
\newblock Minimal almost convexity.
\newblock {\em J. Group Theory}, 8(2):239--266, 2005.

\bibitem{Flajolet}
Philippe {Flajolet} and Robert {Sedgewick}.
\newblock {\em {Analytic combinatorics.}}
\newblock Cambridge: Cambridge University Press, 2009.

\bibitem{TW}
Jennifer Taback and Alden Walker.
\newblock Medium-scale curvature for solvable {B}aumslag-{S}olitar groups,
  2020.
\newblock In preparation.

\end{thebibliography}

\end{document}